\PassOptionsToPackage{numbers}{natbib}
\documentclass[graybox,natbib,envcountsame]{svmult}

\usepackage[T1]{fontenc}
\usepackage{mathptmx}
\usepackage{helvet}
\usepackage{courier}
\usepackage{makeidx}
\usepackage{graphicx}
\usepackage{multicol}
\usepackage[bottom]{footmisc}

\usepackage{amsmath}   
\usepackage{amssymb} 
\usepackage[utf8]{inputenc}
\usepackage{tikz}
\usetikzlibrary{%
cd, 
matrix 
}

\usepackage{mathrsfs}
\usepackage[all]{xy}

\newcommand  \ind[1]  {   {1\hspace{-1.2mm}{\rm I}}_{\{#1\} }    }

\newcommand{\R}{\mathbb{R}}

\newcommand {\al} {\alpha}
\newcommand {\eps}  {\varepsilon}

\newcommand {\dv}  { {\rm div} }
\newcommand {\caln} { {\mathcal N} }

\spnewtheorem*{AAT}{Abel's Addition Theorem}{\bfseries}{\itshape}

\smartqed
\usepackage{etoolbox}
\patchcmd{\qed}{\ifmmode\qedsymbol}{\ifmmode\the\qedsymbol}{}{\foobar}
\patchcmd{\qed}{\hfil\qedsymbol}{\hfil\the\qedsymbol}{}{\foobar}
\qedsymbol{\proofSymbol}

\newcommand{\qedhere}{\tag*{\the\qedsymbol}}

\begin{document}

%
%
\title*{Incompressible limit of porous media equation with chemotaxis and growth}

%
%
\author{Qingyou He \and Hai-Liang  Li  \and  Beno\^\i t Perthame}
\institute{Q.~He \at Sorbonne Universit{\'e}, CNRS, Laboratoire de Biologie Computationnelle et Quantitative, F-75005 Paris, \email{qyhe.cnu.math@qq.com } \and %
H.~Li \at School of Mathematical Sciences,
	Capital Normal University, Beijing 100048, P.R. China, 
\email{hailiang.li.math@gmail.com}
\and
  B.~Perthame \at Sorbonne Universit{\'e}, CNRS, Universit\'{e} de Paris, Inria, Laboratoire Jacques-Louis Lions,  F-75005 Paris, \email{benoit.perthame@sorbonne-universite.fr}}
%
%
\maketitle

\abstract{We revisit the problem of proving the incompressible limit for the compressible porous media equation with Newtonian drift and growth. The question is motivated by models of living tissues development including chemotaxis. We extend the problem, already treated by the authors and several other contributions, in using a simplified approach, in treating dimensions two or higher, and in incorporating the pressure driven growth term. We also complete the analysis with stronger $L^4$  estimates on the pressure gradient. The major difficulty is to prove the strong convergence of the pressure gradient which is obtained here by a new observation on an algebraic relation involving the pressure gradient for weak limits. 
}

\section{Compressible porous media with chemotaxis and growth}
\label{sec:CPM}

\noindent{\bf Setting the problem.} 
Several recent mechanical approaches propose to describe the development of living tissues by flows through a porous media formed by the extra-cellular matrix \cite{MR2471306, MR4233467, MR4327909}. Including chemotaxis, this leads to write the compressible porous media equation combined with the Keller-Segel model
\begin{equation} \label{pmc}
\begin{cases}
\partial_t \rho_{m} - {\rm div} \big[\rho_{m} [\nabla p_m - \nabla \phi_m ]\big]= \rho_{m} G(p_m), \qquad x \in \R^n, \; t \geq 0, 
\\
p_m= (\rho_{m})^m, \qquad \phi_m = \caln \star \rho_{m},  
\end{cases}
\end{equation}
where $\rho_m(x,t)$ denotes the cell  number density, $p_m$ is the pressure which here is taken as an  homogeneous law for simplicity (see~ \cite{DHV20, BenAmar_C_F} for variants and comments on this issue), $G(\cdot)$ is the growth/death rate of cells. The chemotactic term is based on the Newtonian (attractive) potential
\begin{align*} \begin{cases}
\mathcal{N}(x)=-\alpha_n \frac  1 {|x|^{n-2}}, \quad \nabla \mathcal{N}(x)=(n-2) \alpha_n  \frac  x {|x|^{n}} , \quad  - \Delta \mathcal{N} = \delta  \quad  (n\geq 3),
\\
 \mathcal{N}(x)= \; \alpha_2  \ln( |x|), \quad \, \nabla \mathcal{N}(x)= \alpha_2  \frac  x {|x|^{2}} , \qquad  \quad \; \; \; - \Delta \mathcal{N} = \delta  \quad (n=2).
\end{cases}
\end{align*}
However, the interested reader can check that the drift term can be much more general in the class
\begin{equation*}
\nabla \phi (x) = \nabla \phi_0(x) + \int_{\R^n} K(x,y)  \rho (y) dy , 
\end{equation*}
as long as it satisfies at least a control (with compactness), for some $1\leq q < \infty$, 
\begin{equation*}
\| \nabla \phi \|_{L^2(\R^n)} \leq C [1+ \| \rho \|_{L^1(\R^n)} + \| \rho \|_{L^q (\R^n)} ], \quad  m\geq q,
\end{equation*}
see Remark \ref{rk:ape}. All the equations in the present paper are understood in the weak sense as in~\cite{he2022incompressible}.
\\

We consider the incompressible limit, that is the limit $m \to \infty$. Departing from the following version of Eq.~\eqref{pmc},
\begin{equation} \label{pmc2}
\partial_t \rho_{m} - \Delta \frac{m}{m+1}( \rho_m)^\frac{m+1}{m} +  {\rm div} [\rho_{m} \nabla \phi_m ] = \rho_{m} G(p_m), 
\end{equation}
we may formally pass to the limits $\rho_\infty$ of $\rho_m$ , $p_\infty$ of $(p_m)^\frac{m+1}{m}$. We obtain that these limits  satisfy  the equation
\begin{equation} \label{pmi}
\begin{cases}
\partial_t \rho_{\infty} -  \Delta  p_\infty + {\rm div} [\rho_\infty \nabla \phi_\infty] = \rho_{\infty} G(p_\infty), \qquad x \in \R^n, \; t \geq 0, 
\\[5pt]
p_\infty (1- \rho_{\infty})=0, \qquad \rho_\infty \leq 1 , \qquad \phi_\infty = \caln \star \rho_{\infty},  
\end{cases}
\end{equation}
and we show later that the limiting equation can also be written equivalently
\begin{align}  \label{pmi2}
\partial_t \rho_{\infty} -  {\rm div} [\rho_\infty \nabla   p_\infty] + {\rm div} [\rho_\infty \nabla \phi_\infty] = \rho_{\infty} G(p_\infty).
\end{align} 

One can also establish the  so-called {\em complementarity relation}, for almost all $t>0$
\begin{equation} \label{cc}
p_\infty \, \Delta [p_{\infty} - \phi_\infty + G(p_{\infty})]=0,  \qquad x \in \R^n.
\end{equation}
System \eqref{pmi} with \eqref{cc}  is a weak and global form of the geometric Hele-Shaw free boundary problem (see \cite{MPQ_17, KMW_23}).
\\

As always, the equation for the pressure $p_m$ is playing a central role in the analysis; it is written
\begin{equation} \label{pressure}
\partial_t p_{m} - |\nabla p_m |^2 +\nabla p_m. \nabla \phi_m = m p_m [\Delta p_m -\Delta \phi_m + G(p_m)].
\end{equation}
The dominant term on the right hand side of this equation explains formally the complementarity relation~\eqref{cc}, the difficulty being to pass to the limit in the quadratic term $p_m \Delta p_m$ which amounts to prove the strong convergence of $\nabla p_m$. 
\\

\noindent{\bf What is new.}  We complement our previous study~\cite{he2022incompressible} by taking into account the source term on the right hand side of Eq.~\eqref{pmc}, treating all dimensions $n\geq 2$ (rather than $n\geq 3$), allowing more general kernels to define $\phi_m$ and reducing the assumptions on the initial data (in terms of regularity, because we do not use the Aronson-B\'enilan estimate here) and removing the compact support assumption.
\\

The purpose of the present paper is also to give a simple proof of the limit $m \to \infty$ in Eq.~\eqref{pmc} and recover the limit formulations~\eqref{pmi} as well as~\eqref{cc}. The difficulty is to prove strong convergence of the gradient $\nabla p_m$ and we show how to use the recent method proposed in \cite{PrX, LiuXu} and extended in \cite{DouLiuZhou, david:hal-03636939}. We base our analysis of the observation in~\cite{david:hal-03636939} that compensated compactness arguments permit to identify certain limits. This however is not enough and we develop a new idea, namely that the relation $\rho_\infty \nabla p_\infty=\nabla p_\infty$ can be proved a priori.
\\

Let us stress  that we prove the strong convergence of $\nabla p_m$ and not only $\nabla (p_m)^{\frac{m+1}{m}}$ as in the previous works using this method.
\\

\noindent{\bf Related studies.} 
This problem has attracted a lot of attention recently and various methods have been used. The problem of the incompressible limit of the compressible porous media equations with a growth/death term controled by pressure was introduced in~\cite{PQV_14} using strong regularity assumptions and the Aronson-B\'enilan estimate. This strategy has been extended succesfully to various situations as a singular pressure law $p=\frac{\eps \rho}{1-\rho}$ in~\cite{HechtVauchelet2017, DHV20}, chemotaxis and a given drift~\cite{DavidS_21}, coupling with an equation for nutrients~\cite {DnPb21}, cell active motion (including a diffusion) in~\cite{PQTV_14}, 
\\

Another route to study the problem is through viscosity solutions and this was performed in~\cite{KimPS,KimPozar_18}. The case of an aditional drift term is treated in~\cite{kim2019singular, KimZ_21}. Still another approach is based on the obstacle problem, \cite{GKM}, a method which has been extended to bi-stable terms on the right hand side in~\cite{KimMellet23}. Recently a Lagrangian formulation is given in~\cite{jacobs2023lagrangian}.
\\

Because the geometric form of the  Hele-Shaw problem uses the set $\Omega(t)=\{ \rho_\infty(t,\cdot) =1\}$, several authors have studied the question to know, when the initial data is the indicator function of $\Omega(0)$, if this property is propagated even if $\Omega(t)$ has little regularity. For this question it is convenient to work on time integrated variables, which leads to the obstacle problem,  we refer to ~\cite{MPQ_17,CKY_2018,Jacobs_Kim_Tong23,KMW_23}. The regularity and  stability of such patches is studied in~\cite{KimLelmi23,CollinsJacobsKim}.
\\

In the conservative case, that is $G\equiv 0$, the Hele-Shaw problem has been handled under the name of {\em congested flows} by \cite{MRS10,MRSV11,MRS14}, using  a new method based on Wasserstein distance and optimal transport. See also for systems \cite{MRS14, Laborde}. From this point of view, the pressure $p_\infty$ is interpreted as a Lagrange multiplier associated with the constraint $\rho_\infty \leq 1$.The derivation departing from the compressible porous media equation has been achieved in~\cite{AKY_14,CKY_2018}. These papers combine the gradient flow approach with methods based on viscosity solutions which also allow to define another notion of solutions. 
\\

Eq. \eqref{pmc} assumes that cells are pushed passively, when including active movement, one arrives at the equation 
\begin{align*} 
\partial_t \rho_{m} -  \mu \Delta \rho_m -{\rm div} \big[\rho_{m} [\nabla p_m - \nabla \phi_m ]\big]= \rho_{m} G(p_m).
\end{align*} 
This problem is studied in \cite{PQTV_14} and in the conservative case~\cite{KMW_23}. There is still a set $\Omega(t)$ where $p_\infty >0$, but $\rho_\infty $ is smooth and has a positive tail because the equation is parabolic non-degenerate. 
\\

Error estimates for $p_m - p_\infty$ have also been established in~\cite{AKY_14, CKY_2018} for the conservative case with Newtonian potential in Wasserstein distance.  With a growth term (and no drift) error estimates in $H^{-1}$ are proved in~\cite{DDP22}.
\\

Related to the limit $m \to \infty$, we also mention that the growth term allows for traveling waves. The existence and Hele-Shaw limit are established in~\cite{dalibard:hal-03300624}, see also the references therein and, for the problem with a necrotic core \cite{DouZhouNecrotic2023}. Traveling waves in the case with nutrients are also built analytically in~\cite{PTV14}. Also Darcy's law for the prorous media can be generalized to other rules as Brinkman's law, see~\cite{DPS21,DDMS23,KimTuranova_18}, or as full Navier-Stokes system~\cite{Vauchelet_Z}.
\\

\noindent{\bf Outline of the paper.} In the next section, we present the basic estimates necessary to analyze the problem. With these estimates we can establish, in Section~\ref{sec:limit}, several relations between weak limits of $\rho_m$ and $p_m$. Those are used in Section~\ref{sec:limites} to prove that $\nabla p_m$ converges strongly in $L^2$ and this establishes the complementarity condition. In Section~\ref{sec:regularity}, we complete our analysis with some additional bounds on $p_m$ in particular $\nabla p_m \in L^4$.

\section{Estimates}
\label{sec:main}

We complete system \eqref{pmc} with an initial data  $\rho^0_m \geq 0$  and set $p^0_m= (\rho^0_m)^m$. We assume that for some constant $K^0$ independent of $m$,
\begin{equation}  \label{as:ID}
\int_{\R^n}  \rho^0_m  \big[1 +|x|^2\big] dx  \leq K^0, \qquad  \int_{\R^n} p^0_m [1+ (\rho^0_m)^4]  dx \leq K^0 m.
\end{equation}

For the right hand side of Eq.~\eqref{pmc}, we assume there is an 'homeostatic pressure' $p_H >0$  and a constant $G_M$ such that
\begin{equation}  \label{as:G}
G\in C^2([0, \infty); \R), \qquad G(p) \leq 0 \quad \text{for} \quad p\geq p_H, \qquad | G(\cdot) | \leq G_M.
\end{equation}

We recall standard estimates.
\begin{proposition}[A priori estimates] \label{thm:ape} 
With the assumptions \eqref{as:ID}, \eqref{as:G}, the solution of Eq.~\eqref{pmc} satisfies the following bounds for constant $C(T)$ independent of $m$,  and for $t\leq T$, 
\\[5pt]
(i) $\ \;   \;  \int_{\R^n} \rho_m (t) dx \leq M(t):=  e^{tG(0)} \int_{\R^n}  \rho^0_m dx$,
\\[5pt]
(ii) $ \; \int_0^T\int_{\R^n} |\nabla p_m (t)|^2 dx dt \leq C(T)$, \qquad $\int_{\R^n} p_m(t) dx \leq C(T)m$, 
\\[5pt]
(iii) $ \;  \int_0^T\int_{\R^n} \big[ |\nabla p_m^\frac{m+1}{m} |^{2}+ |\nabla p_m^\frac{m+2}{m} |^{2} \big] dx dt \leq C(T)$,  \quad $\int_0^T \int_{\R^n}  (\rho_m)^{m+5}  dx dt \leq C(T)$,
\\[5pt]
(iv) $ \;  \int_{\R^n} | x|^2 \rho_m (t) dx   \leq C(T)$,
\end{proposition}

As usual, several further conclusions follow from this proposition.  For instance, the second estimate in (ii) and interpolation with (i) give for $0\leq t \leq T$, 
\begin{align*}
\| \rho_m (t) \|_{L^m(\R^n)} \leq (C(T)m)^{\frac 1 m} \leq C(T), \qquad \| \rho_m (t) \|_{L^q(\R^n)} \leq C(T), \; 1\leq q \leq m.
\end{align*}
Furthermore, the first estimate in (ii), (iii) and the Sobolev inequality also give
\begin{align} \label{est2} \begin{cases}
 \|  p_m+ (p_m)^\frac{m+1}{m}  \|_{L^2\big((0,T); L^{\frac {2n}{n-2}}(\R^n)\big)} \leq C(T), \qquad n\geq 3,
 \\
  \|  p_m \|_{L^r\big((0,T); L^q(\R^n)\big)} \leq C(T), \; \frac{m+1}m \leq q < \infty,  \; \frac 2 r=1+ \frac{m+1}{mq}, \quad n=2, 
\end{cases}\end{align}
where, for $n=2$, the inequality follows from the Gagliardo-Nirenberg-Sobolev inequality and the third estimate in (ii).
\\

These estimates on $\rho_m\in L^\infty((0,T); L^1\cap L^m\big(\R^n)\big)$ also have consequences on the field $\phi_m$ when it is given by the Newtonian field. Using the Young inequalities for convolutions, or Sobolev inequalities, we obtain, for $0\leq t \leq T$, 
\begin{align} \label{gradphi}
\| \nabla \phi_m (t) \|_{L^q(\R^n)} \leq C(T), \qquad  \frac {d}{d-1}< q \leq \infty, \quad \text{for} \; m>n.
\end{align}
Also because
\begin{align*}
\partial_t \phi_m (t) = - \frac{m}{m+1}( \rho_m)^\frac{m+1}{m} - \Delta^{-1}  {\rm div} [\rho_{m} \nabla \phi_m ] +\Delta^{-1}   \rho_{m} G(p_m),
\end{align*}
we conclude bounds as
\begin{align}\label{dtphi}
\partial_t \phi_m (t) \quad \text{is bounded in}  \quad L^2\big((0,T); L^{\frac {2n}{n-2}}(\R^n)\big).
\end{align}
These further estimates are not necessarily sharp but are enough for our purposes in the next sections. Lemma~\ref{lm:interpol} furnishes other estimates by interpolation.

\begin{proof}
The first bound is immediate by integrating Eq. \eqref{pmc}. 
\\

For the bound (ii), we compute, for $t \leq T$, integrating Eq.~\eqref{pressure}
\begin{align*}
  \frac{d}{dt}   \int_{\R^n} p_m(t) dx  +\int_{\R^n} (m-1) |\nabla p_m|^2 dx
  = \int_{\R^n}\big[  (m-1) \nabla p_m. \nabla \phi_m + m p_m G(p_m) \big]dx.
\end{align*}
Because $ \min(p_m, p_H) \leq (p_H)^{\frac {m-1} m} \rho_m$, and integrating by parts the potential term, we find
\begin{align*}
\frac{d}{dt}   \int_{\R^n} p_m(t) dx + \int_{\R^n} (m-1) |\nabla p_m|^2 dx \leq (m-1)    \int_{\R^n} p_m \rho_m  dx + m  p_H^{\frac {m-1} m} M(T) G_M .
\end{align*}
This shows the estimates (ii). 
\\

The estimate (iii) is just a variant of (ii) obtained multiplying Eq.~\eqref{pressure} by the power $(\rho_m)^2 =(p_m)^{\frac{2}{m}}$ or $(\rho_m)^4 =(p_m)^{\frac{4}{m}}$, and integrating by parts. We find, for instance for the first case
\begin{align*}
\frac{m}{m+2}  \frac{d}{dt}   \int_{\R^n} p_m(t)^\frac{m+2}{m} & dx  + \int_{\R^n} (1+m) (\rho_m)^2 |\nabla p_m|^2 dx
\\
&  = \int_{\R^n}\big[ m \frac{m+1}{m+2} (\rho_m)^{m+3} + m p_m \rho_m^2 G(p_m) \big]dx.
\end{align*}
As before, we obtain the  inequality
\begin{align*}
\frac{m}{m+1} \frac{d}{dt}  \int_{\R^n} p_m(t)^\frac{m+2}{m} dx  +& \frac{ m^2}{1+m} \int_{\R^n} |\nabla (p_m)^\frac{1+m}{m}|^2 dx
\\
&  \leq m \frac{m+1}{m+2}  \int_{\R^n} (\rho_m)^{m+3}  dx + m  p_H^\frac{m+1}{m} M(T) G_M .
\end{align*}
Then, we use \eqref{est:1} with $\widetilde m=m+ 1 $ and $k=1$ to control the term $(\rho_m)^{m+3}$ by the term $ |\nabla (p_m)^\frac{1+m}{m}|^2$ and we conclude the first estimate of (iii).  The same computation gives the estimate on $|\nabla (p_m)^\frac{2+m}{m}|^2 $ and on  $(\rho_m)^{m+5}$.

The bound on $(\rho_m)^{m+2} $ is also a conclusion of \eqref{est:1} with $\widetilde m =m+2$ and $k=0$.
\\

For proving (iv), we compute
 \begin{align*}
\frac{d}{dt} \int_{\R^n} \frac{| x|^2}{2}  \rho_m dx &\leq - \int_{\R^n} \rho_m x.[m(\rho_m)^{m-1} \nabla \rho_m -\nabla \phi_m] dx + G(0) \int_{\R^n} | x|^2 \rho_m dx
\\
& \leq   \frac{nm}{m+1}  \int_{\R^n} (\rho_m)^{m+1} dx - \frac{n-2}{2}  \alpha_n  \int_{\R^{2n}}  \frac  1  {|x-y|^{n-2}} \rho_m (x)\rho_m(y) dx dy 
\\
& \qquad  + G(0) \int_{\R^n} | x|^2 \rho_m dx.
\end{align*}
It remains to use the second bound in (iii) and Proposition~\ref{thm:ape} is proved.

\end{proof}

For this proof, we use the following general result.
\begin{lemma} \label{lm:interpol}
Let  $\rho \geq 0$ and set $M:=\int_{\R^n} \rho (x) dx<\infty $.  Assume that for some ${\widetilde m}>1$, $ \nabla \rho^{\widetilde m} \in L^2(\R^n)$. Then, for all $0 \leq \frac{k+1}{\widetilde m} < 1$
\begin{align} \label{est:1}
\int_{\R^n}  \rho(x)^{{\widetilde m}+k+1}  dx \leq C(\frac {k+1}{\widetilde m},n,M) + \frac 1 2 \int_{\R^n} |\nabla \rho(x)^{\widetilde m}|^2 dx.
\end{align}
The constant $C(\frac {k+1}{\widetilde m},n,M)$ blows-up as $\frac{k+1}{\widetilde m}\to 1$, $ {\widetilde m}\to \infty$ and is unifornly bounded on each closed subinterval.
\end{lemma}

The difficulty here is to obtain controls independent of $m$, which forbids  direct use of the Sobolev inequalities on $\rho$. 
\begin{proof}
For a constant $A$ to be chosen later, we decompose $ \rho$ as follows
 \begin{align*} 
 \rho= \min (A^{\frac 1 {{\widetilde m}+k}}, \rho) +  \psi(\rho), \qquad  \psi'(\rho)=0 \quad \text{for}\;  \rho <A^{\frac 1 {{\widetilde m}+k}}, \qquad 0\leq \psi'(\rho)\leq 1.
 \end{align*} 
 Then, we use the Gagliardo-Nirenberg-Sobolev inequality for $u= \psi(\rho)^\frac{{\widetilde m}+k+1}{2}$  and the Cauchy-Schwarz inequality. We find distinguishing the subsets where $\rho \leq A^{\frac 1 {{\widetilde m}+k}}$ or not, 
\begin{align*}
\int_{\R^n}  \rho^{{\widetilde m}+k+1}  dx& \leq A M+  C(n)\left( \int_{\R^n} \psi(\rho)^\frac{{\widetilde m}+k+1}{2} dx\right)^{\frac{4}{n+2}} 
 \left( \int_{\R^n} \big| \nabla \psi(\rho)^\frac{{\widetilde m}+k+1}{2} \big|^2 dx\right)^{\frac{n}{n+2}}
\\
&\leq AM+  C(n) \left( \int_{\R^n} \psi(\rho)^{{\widetilde m}+k+1} dx \;  {\rm meas} (\{ \rho >A^{\frac 1 {{\widetilde m}+k}}\} \right)^{\frac{2}{n+2}} 
\\
& \qquad \qquad  \qquad \qquad 
 \left( \int_{\R^n} \big| \frac{{\widetilde m}+k+1}{2}   \psi(\rho)^\frac{{\widetilde m}+k-1}{2} \psi'(\rho) \nabla \rho \big|^2 dx\right)^{\frac{n}{n+2}}.
 \end{align*}
Because $A^{\frac 1 {{\widetilde m}+k} } {\rm meas} (\{ \rho >A^{\frac 1 {{\widetilde m}+k}}\} \leq M$,  the last term  can be further controled as 
\begin{align*}
 C&(n)\left( A^{\frac {-1} {{\widetilde m}+k}} M \int_{\R^n} \rho^{{\widetilde m}+k+1} dx \right)^{\frac{2}{n+2}} 
 \left( \int_{\R^n} \big| \frac{{\widetilde m}+k+1}{2} A^{-\frac{{\widetilde m}-k-1}{2({\widetilde m}+k)}}  \rho^{{\widetilde m}-1} \nabla \rho \big|^2 dx\right)^{\frac{n}{n+2}}
  \\
&\leq A^{-\frac{2+n({\widetilde m}-k-1)}{({\widetilde m}+k)(n+2)}}C(n) \left( M \int_{\R^n} \rho^{{\widetilde m}+k+1} dx \right)^{\frac{2}{n+2}} 
 \left( \int_{\R^n} \big| \frac{{\widetilde m}+k+1}{2{\widetilde m}} \nabla \rho^{\widetilde m} \big|^2 dx\right)^{\frac{n}{n+2}}
\end{align*}
because $\psi(\rho) >0$ only when $\rho > A^{\frac 1 {{\widetilde m}+k}}$.
Consequently, choosing $A$ large enough independently of ${\widetilde m}$, we may write.
\begin{align*}
\int_{\R^n}  \rho^{{\widetilde m}+k+1}  dx \leq  A M + \left( \int_{\R^n} \rho^{{\widetilde m}+k+1} dx \right)^{\frac{2}{n+2}} 
 \left( \frac 1 2 \int_{\R^n} \big|  \nabla \rho^{\widetilde m} \big|^2 dx\right)^{\frac{n}{n+2}}.
\end{align*}
This directly gives the estimate \eqref{est:1}.
\end{proof}

\begin{remark} \label{rk:ape} Another way to perform the estimate (ii) in Propositon \ref{thm:ape} is as follows. The control of $\int_{\R^n}  p_m dx$ is changed as 
\begin{align*}
\frac{d}{dt}   \int_{\R^n} p_m(t) dx +& \int_{\R^n} (m-1) |\nabla p_m|^2 dx 
\\
&\leq (m-1)    \int_{\R^n} \nabla p_m .\nabla \phi_m  dx + m  p_H^{\frac {m-1} m} M(T) G_M 
\\
&\leq \frac{m-1}{2}     \int_{\R^n}[ | \nabla p_m|^2 + |\nabla \phi_m |^2] dx + m  p_H^{\frac {m-1} m} M(T) G_M .
\end{align*}
Consequently any control of  $\int_{\R^n}  |\nabla \phi_m |^2 dx$ by a fixed norm $L^q$ of $\rho_m$ will allow to close the bound thanks to~\eqref{est:1} for $m \geq q$ (not optimal in view of \eqref{est2}).
\end{remark}

\begin{remark} \label{rk:gradphalf} Because we also have $ \int_{\R^n} (p_m^0)^{\frac  {m+1}{2m}} \leq K_0 m^{\frac 1 2} $, a similar argument  as in the proof of (ii) arrives to the conclusion that we can control lower powers of $p_m$
\begin{align*}
 \int_0^T\int_{\R^n} |\nabla p_m (t)^{\frac  {m+1}{2m}}|^2 dx dt \leq C(T).
 \end{align*}
\end{remark}

\section{Fundamental relations for the weak limits}
\label{sec:limit}

Departing from the integrability properties in Proposition~\ref{thm:ape}, after extraction, there are weak limits
\begin{align}\label{wlim1} \begin{cases}
\rho_m \rightharpoonup \rho_\infty  \quad & \text{in } w-L^q((0,T)\times \R^n), \qquad 1\leq q < \infty,
\\[2pt]
\rho_m p_m \rightharpoonup p_\infty \quad & \text{in } w-L^2((0,T);  \dot H^1(\R^n) .
\end{cases}
\end{align}
For the second statement, the existence of a weak limit is a consequence of the bound on $\rho_m p_m$ stated in Proposition~\ref{thm:ape} (iii). We call it $p_\infty$ because of one of the statements below.

Passing to the weak limit in the equation under the form \eqref{pmc2}, we obtain, as in Eq.~\eqref{pmi}, the relation  between $ \rho_\infty$ and $p_\infty$
\begin{align} \label{eq:pinfty}
\partial_t \rho_{\infty} - \Delta p_\infty + {\rm div} [\rho_{\infty} \nabla \phi_\infty ]=R(t,x) \in L^\infty\big((0,T); L^1\cap L^\infty (\R^n)\big). 
\end{align}
Indeed, 
\\
$\bullet$ the drift term  $\rho_m \nabla \phi_m$ passes to the limit by weak-strong limit since $\nabla \phi_m$ converges strongly to  $\nabla \phi_\infty$ because of \eqref{gradphi}, \eqref{dtphi} and $D^2 \phi_m \in L^q$ by singular integral theory (see also~\cite{he2022incompressible} for details). 
\\
$\bullet$ the right hand side passes to the weak limit because $G$ is bounded and thus $\rho_m G(p_m)$ is bounded in the same Lebesgue spaces as $\rho_m$, following the bounds in Proposition~\ref{thm:ape}. In fact, by compensated compactness, \cite{MuratCC, TartarCC} and  following~\cite{david:hal-03636939}, and as it is also used to prove (iii) in the proposition below, it follows that
\begin{align*}
R(t,x)= \rho_\infty G(p_\infty).
\end{align*}

Also, we insist that the weak limits of $p_m$ and $\rho_m p_m$ are the same, as stated in the
\\

\begin{proposition} [Fundamental relations] \label{fundrel} 
  With the assumptions of Proposition~\ref{thm:ape}, these limits satisfy 
\\[3pt]
(i) $\qquad \; \; \,  \rho_\infty \leq 1$ for almost all $t\in (0,T)$ and $x \in \R^n$, 
\\[3pt]
(ii) $\qquad  \ p_m \rightharpoonup p_\infty \quad \text{in } w-L^p((0,T)\times \R^n)$,
\\[3pt]
(iii) $\qquad p_\infty= \rho_\infty p_\infty$ for almost all $t\in (0,T)$ and $x \in \R^n$,
\\[3pt]
(iv) $\qquad  \, \rho_\infty \nabla p_\infty = \nabla p_\infty \in L^2 \big((0,T) \times \R^n\big)^n$,
\\[3pt]
(v) $\qquad \; \partial_t \rho_\infty \, p_\infty =0$, \qquad in duality $L^2\big((0,T);H^{-1}(\R^n)\big)$, $L^2\big((0,T);H^{1}(\R^n)\big)$.
\end{proposition}

The new information that we bring here is the direct proof of the identity in (iv) based on the estimates of Proposition~\ref{thm:ape} only. It is instrumental to treat the drift term. Also it establishes that both equations \eqref{pmi} and \eqref{pmi2} hold.

Notice also that the product $ \partial_t \rho_\infty \; p_\infty $ is well defined  by duality $L^2_t (H^{-1})$, $L^2_t (H^{1})$.

\begin{proof}
For (i), we recall that the inequality $ \rho_\infty \leq 1$ is a consequence of the bound in Proposition~\ref{thm:ape}~(ii)  since 
\begin{align*}
\| \rho_m (t) \|_{L^{m}(\R^n)} \leq (C(T)m)^\frac 1 m   \underset{ m \to \infty }{\longrightarrow} 1, \qquad 0 \leq t \leq T,
\end{align*}
which implies  that for all $1 \leq q\leq m$, by interpolation between $L^1$ and $L^m$,
\begin{align*}
\| \rho_m (t) \|_{L^{q}(\R^n)} \leq M(T)^{1-\theta} \big(mC(T)\big)^{\frac \theta m} , \qquad \frac 1 q= (1-\theta) + \frac \theta m.
\end{align*}
In the weak limit we find 
\begin{align*}
\| \rho_\infty (t) \|_{L^{q}(\R^n)} \leq M(T)^{\frac 1 q} \to 1 \quad \text{as } q \to \infty.
\end{align*}

Next, we prove (ii), i.e., that the weak limits of $p_m$ and $\rho_m p_m$ are the same. This follows from the two inequalities
\begin{align*}
p_m \leq  (\rho_m)^m = \frac {1}{m+1}+ \frac {m}{m+1}(\rho_m)^{m+1},
\\
(\rho_m)^{m+1} \leq \frac{m+1}{m}(\rho_m)^{m} +\frac{m+1}{2m^2} (\rho_m)^{2m},
\end{align*}
thanks to  the bounds in Proposition~\ref{thm:ape}  and estimate \eqref{est2}.
\\

For (iii), i.e., the identity $p_\infty= \rho_\infty p_\infty$ is a consequence of the above convergence and compensated compactness, following~\cite{david:hal-03636939} 
\begin{align} \label{p=rhop}
 \rho_\infty p_\infty =w-\lim \;  \rho_m p_m
\end{align}
because $\nabla p_m$ is bounded in $L^2$, see Proposition~\ref{thm:ape} (ii), and $\partial_t \rho_m $ is bounded in $L^2_t(H^{-1})$ thanks to Eq.~\eqref{pmc2} and  the bound in  Proposition~\ref{thm:ape} (ii)  (see \cite{he2022incompressible,david:hal-03636939} for details and other arguments).
\\

For proving (iv),  we use Proposition~\ref{thm:ape} (ii) and that the chain rule holds in $W_{\rm loc}^{1,1}$. For $A$ large, consider a non-decreasing 'truncation' function which satisfies $\chi_A(p) =p$ for $0\leq p \leq \frac A 2$, $\chi_A(p) =A$ for $p\geq 2A$.   We may write
\begin{align*}
p_\infty = \chi_A(p_\infty) + [p_\infty-\chi_A(p_\infty)].
\end{align*}
On the one hand, because $\chi_A(p_\infty)$ is bounded, the mapping $\chi \mapsto \chi^{1+\eps}$ is Lipschitz and we may use the chain rule  to write
\begin{align*}
\nabla \chi_A(p_\infty)^{1+\eps} = (1+\eps) \chi_A(p_\infty)^{\eps}\nabla \chi_A(p_\infty)= \rho_\infty (1+\eps) \chi_A(p_\infty)^{\eps}\nabla \chi_A(p_\infty),
\end{align*}
still using (iii) and because when $\rho_\infty \neq 1$, then $p_\infty=\chi_A(p_\infty)^{\eps} =0$. Therefore, we have obtained that
\begin{align*}
\nabla \chi_A(p_\infty)^{1+\eps} =\rho_\infty \nabla \chi_A(p_\infty)^{1+\eps}
\end{align*}
and as $\epsilon \to 0$ we find  $ \nabla \chi_A(p_\infty) = \rho_\infty \nabla \chi_A(p_\infty)$.  
\\

On the other hand, $\nabla p_\infty- \nabla \chi_A(p_\infty) \to 0 $ in $L^2\big((0,T)\times \R^n \big)$ as $A \to \infty$. 
All together, we have proved (iv).
\\

The identity (v) can be obtained in different ways, see \cite {PrX,LiuXu,Igbida}.  A  simple  argument is given in~\cite{david:hal-03636939} and uses that $\partial_t \rho_\infty \in L^2\big( (0,T); H^{-1}(\R^n)\big)$ (see Eq.~\eqref{eq:pinfty})  and $p_\infty\in L^2\big( (0,T); H^{1}(\R^n)\big)$ (see Proposition~\ref{thm:ape} (ii)).  Therefore $\partial_t \rho_\infty (t) \, p_\infty (t)$ can be approximated successively as
\begin{align*}
\frac{\rho_\infty (t+h) -\rho_\infty (t)}{h} p_\infty (t) =\frac{\rho_\infty (t+h) -1}{h} p_\infty (t) \leq 0, 
\\
\frac{\rho_\infty (t) -\rho_\infty (t-h)}{h} p_\infty (t) =\frac{1-\rho_\infty (t-h)}{h} p_\infty (t) \geq 0,
\end{align*}
where we have only used again (iii).
\end{proof}

\section{Stong convergence of $\nabla p_m$ and the complementarity condition}
\label{sec:limites}

A simple procedure to pass to the limit $m \to \infty$ and recover the complementarity relation~\eqref{cc} has been elaborated when chemotaxis is ignored. We follow the most advanced form in~\cite{david:hal-03636939}, based on the ideas  in~\cite{PrX,LiuXu,DouLiuZhou}.

The first step is to  pass to the limit in Eq.~\eqref{pmc2}  and, as already established in Section~\ref{sec:limit}, we find that
\begin{align} \label{pmclimit}
\begin{cases}
\partial_t \rho_{\infty} - \Delta p_\infty + {\rm div} [\rho_{\infty} \nabla \phi_\infty ]= \rho_{\infty} G(p_\infty), \qquad x \in \R^n, \; t \geq 0, 
\\
p_{\infty}= \rho_{\infty} p_{\infty}, \qquad \phi_{\infty} = \caln \star \rho_{\infty},  \qquad \rho_{\infty} \leq 1.
\end{cases}
\end{align}
Thanks to the relations proved in Propostion~\ref{fundrel}, one can establish the
\begin{theorem} [Strong convergence of $\nabla p_m$] \label{thm:strong} 
With the assumptions of Proposition~\ref{thm:ape}, we have after extraction of a subsequence
\begin{align*}
 \nabla (p_m)^\frac{m+1}{m} \to \nabla p_\infty \qquad &\text{in } L^2\big((0,T)\times \R^n\big). 
\end{align*}
Therefore we have
 \begin{align*}
 (p_m)^\frac{m+1}{m} \to p_\infty \qquad &\text{in } L^2\big((0,T)\; L^{\frac{2n}{n-2}}( \R^n)\big), \qquad n \geq 3,
\end{align*}
and this holds for $n=2$ in interpolated Lebesgue spaces. 
\\

Assuming additionally $ \int_{\R^n} (p_m^0)^{\frac  1 2} \leq K_0 m$, we also have 
\begin{align*}
 \nabla p_m \to \nabla p_\infty \qquad &\text{in } L^2\big((0,T)\times \R^n\big). 
\end{align*}
Therefore the complementarity relation \eqref{cc} holds true.
\end{theorem}


\begin{proof} {\bf Strong convergence of $ \nabla (p_m)^\frac{m+1}{m}$.} We substract Eq.~\eqref{pmclimit} to Eq.~\eqref{pmc} and we find
\begin{align*} 
\partial_t (\rho_m - \rho_{\infty}) - \Delta( \frac{m}{m+1} \rho_m^{m+1} - p_{\infty} ) +\dv & (  \rho_m \nabla \phi_m -\rho_\infty \nabla \phi_\infty) 
\\
&=  \rho_m G(p_m) -  \rho_\infty G(p_\infty).
\end{align*}
Multiplying by $\frac{m}{m+1} \rho_m^{m+1} - p_{\infty} $ and integrating by parts gives
\begin{align*}
&  \int_{\R^n} (\frac{m}{m+1} \rho_m^{m+1} - p_{\infty} )  \partial_t (\rho_m - \rho_{\infty})]dx +  \int_{\R^n} |\nabla( \frac{m}{m+1} \rho_m^{m+1} - p_{\infty} )|^2  dx
 \\
 &-  \int_{\R^n} ( \nabla( \frac{m}{m+1} \rho_m^{m+1} - p_{\infty} ) (  \rho_m \nabla \phi_m -\rho_\infty  \nabla\phi_\infty)dx
=  \int_{\R^n}[\rho_m G(p_m) -  \rho_\infty G(p_\infty)] dx.
\end{align*}
This serves to show that the term $\int_{\R^n} |\nabla( \frac{m}{m+1} \rho_m^{m+1} - p_{\infty} )|^2  dx$ converges to $0$ as $m\to \infty$ because the three other terms do so.

As already mentioned, and following \cite{david:hal-03636939}, the last term vanishes as $m\to \infty$ by compensated compactness.

The first term (with $\partial_t$) vanishes also as in  \cite{PrX,LiuXu,david:hal-03636939}. This is a simple consequence of the identity $p_{\infty} \, \partial_t \rho_{\infty}=0$ stated in Propostion~\ref{fundrel} (v).

 The new difficulty comes from the drift term. We consider successively the four terms of this product and show we can pass to the limit. For the first product, we consider
\begin{align*}
\int_{\R^n}  \nabla \frac{m}{m+1} \rho_m^{m+1}\rho_m \nabla \phi_m  dx=\int_{\R^n}  \frac{m}{m+2}\nabla \rho_m^{m+2} \nabla \phi_m dx \to \int_{\R^n} \nabla p_{\infty} \nabla\phi_\infty dx,
\end{align*}
here is the only place where we use estimate $\nabla \rho_m^{m+2} $ in Proposition~\ref{thm:ape} (iii).

The second product we consider is 
\begin{align*}
\int_{\R^n} \nabla p_{\infty} \rho_m \nabla \phi_m dx \to \int_{\R^n} \nabla p_{\infty} \rho_\infty \nabla \phi_\infty dx
\end{align*}
where the convergence is obtained by weak-strong limits because $\nabla \phi_m$ converges strongly as used before (or again by compensated compactness).

The third product is 
\begin{align*}
\int_{\R^n}  \nabla \frac{m}{m+1} \rho_m^{m+1} \rho_\infty  \nabla\phi_\infty dx \to  \int_{\R^n} \nabla p_{\infty} \rho_\infty \nabla \phi_\infty dx
\end{align*}
as a weak limit tested agianst a given function.

The fourth product does contain terms which  pass to the limit and thus, we conclude the limit
\begin{align*}
 \int_{\R^n} ( \nabla( \frac{m}{m+1} \rho_m^{m+1} - p_{\infty} ) (  \rho_m \nabla \phi_m -\rho_\infty  \nabla\phi_\infty)dx
 \to \int_{\R^n} \nabla p_{\infty} (\rho_\infty -1) \nabla \phi_\infty dx =0
\end{align*}
thanks to Proposition \ref{fundrel}  (iv). 

This concludes the strong convergence of $\nabla (p_m)^\frac{m+1}{m}$.
\\

\noindent {\bf Strong convergence of $ (p_m)^\frac{m+1}{m}$.} This is just a consequence of the Sobolev inequality (we only deal with the dimensions larger than 3)
\begin{align*}
\|(p_m)^\frac{m+1}{m} - p_\infty \|_{L^2\big((0,T)\; L^{\frac{2n}{n-2}}( \R^n)\big)} \leq \| \nabla (p_m)^\frac{m+1}{m} - \nabla p_\infty \|_{L^2\big((0,T)\times \R^n\big)} .
\end{align*}

\noindent {\bf Strong convergence of $ \nabla p_m $.} 
We fix a parameter $0< \eps \leq \frac 12$ and decompose 
\begin{align*}
 \nabla (p_m)^\frac{m+1}{m} -\frac{m+1}{m}  \nabla p_m =2 \frac{m+1}{m} & \nabla (p_m)^{\frac 12} \ind{\rho_m \leq 1-\eps} [(p_m)^{\frac 12}- (p_m)^{\frac 12+ \frac 1 m}]
 \\
 &+  \nabla (p_m)^\frac{m+1}{m}  \ind{\rho_m > 1-\eps} \frac{1-\rho_m}{\rho_m}.
\end{align*}
From this, we infer
\begin{align*}
 |\nabla (p_m)^\frac{m+1}{m} -\frac{m+1}{m}  \nabla p_m| \leq &4 |\nabla (p_m)^{\frac 12}| (1-\eps)^{\frac m2}
 \\
 &+  | \nabla (p_m)^\frac{m+1}{m} | [ 2 \eps + \ind{\rho_m > 1+\eps} \frac{\rho_m-1}{\rho_m}].
\end{align*}
As $m \to \infty$, the  first term converges to $0$ in $L^2$ thanks to the bounds in Remark~\ref{rk:gradphalf} and in Proposition~\ref{thm:ape} (ii), (iii). To treat the last one, we analyze the quantity $Q_m$ defined as  
\begin{align*}
 Q_m:=\ind{\rho_m > 1+\eps} \frac{\rho_m-1}{\rho_m}, \quad  0 \leq Q_m\leq 1,  \quad Q_m \leq \rho_m \in L_{\rm bounded}^\infty \big((0,\infty); L^1(\R^n)\big).
\end{align*}
We also have, since $(p^m)^{\frac{m+1}{m}} $ is bounded in $L^1\big((0,T)\times \R^n\big)$, 
\begin{align*}
 0\leq Q_m \leq \frac{(p^m)^\frac{m+1}{m}} {(1+\eps)^m} \to 0  \quad \text{a.e.  as } m\to \infty.
 \end{align*}
 Therefore $Q_m \to 0$ in $L^2\big((0,T)\times \R^n\big)$.
 
 Consequently, still using Proposition~\ref{thm:ape} (iii), we have obtained that, for all $\eps$ with $0< \eps \leq \frac 12 $
 \begin{align*}
\limsup_{m\to \infty} \|  \nabla (p_m)^\frac{m+1}{m} -\frac{m+1}{m}  \nabla p_m \|_{L^2\big((0,T)\times \R^n\big)} \leq  2 \eps C(T).
 \end{align*}
 which proves the stong convergence of $\nabla p_m$.
 \\
 
 Then we can pass to the limit in \eqref{pressure} and obtain the complementarity relation.
\end{proof}
\section{More regularity}
\label{sec:regularity}

Several further regularity results can easily be obtained in the context of our assumptions. 
\\

\noindent {\bf Maximum principle for $p_m$.} 
We complement the assumptions \eqref{as:G} on  $G(\cdot)$ by 
\begin{equation} \label{as:Gstrong}
\exists P_M >0 \quad \text{such that} \quad  A+ G(P_M) \leq  0 \quad \forall A \in [1, P_M].
\end{equation}

We also assume here that $\phi_m$ satisfies the bound (obvious for the Newtonian field)
\begin{equation} \label{as:phistrong}
- \Delta \phi_m  \leq \| \rho_m (t)\|_{L^\infty(\R^n)}.
\end{equation}

\begin{lemma} \label{lm:linfty}
We assume \eqref{as:ID}, \eqref{as:G}, \eqref{as:Gstrong}, \eqref{as:phistrong} and $p_m^0 \leq P_M$, then 
$p_m(t) \leq P_M$ for all $t>0$.
\end{lemma}
\begin{proof}
Departing from Eq. \eqref{pressure}, we may write
\begin{align*}
\partial_t p_{m} + \nabla p_m. [ \nabla \phi_m - \nabla p_m] -  m p_m \Delta p_m &= m p_m [- \Delta \phi_m + G(p_m)]
\\
& \leq
p_m \big[  \| p_m (t)\|^\frac 1 m _{L^\infty(\R^n)} + G(p_m) \big] .
\end{align*}
Because $A:=P_M^\frac 1 m \in  [1, P_M]$, the maximum principle applies and we obtain the conclusion. 
\end{proof}

\noindent {\bf The estimate $\nabla p_m \in L^4$.} 
We continue with an extension of  the $L^4$ estimate introduced in \cite{DnPb21,AlazardBresch}, see also \cite{DavidSantambrogio23} for a recent Lipschitz bound. To simplify we only consider the case
\begin{align*}
- \Delta \phi_m = \rho_m.
\end{align*}

\begin{theorem} [$L^4$ estimate for $\nabla p_m$]
With the assumptions \eqref{as:ID}, \eqref{as:G}, and $|\nabla  p_m^0|$ is bounded in $L^2(\R^n)$, we have for all $T>0$ and some constant $C(T)$  independent of $m$
\begin{align} \label{est:s1}
\int_{\R^n}|\nabla  p_m (T)|^2  dx+m \int_0^T \int_{\R^n} p_m &\big(\Delta p_m + \rho_m +G(p_m)\big)^2 dx dt 
\\
& + \int_0^T \int_{\R^n} p_m |D^2 p_m|^2 dx dt  \leq C(T). \notag
\end{align}
Additionally, if $p_m$ is bounded in $L^\infty$ (see Lemma \ref{lm:linfty}), then for $0\leq \al <1$, 
\begin{align} \label{est:s2}
(1-\al)^2 \int_0^T \int_{\R^n} \frac{|\nabla p_m |^4 }{p_m^\al} dx dt \leq  \| p_m \|^{1-\al}_{L^\infty((0,T)\times \R^n)}C(T).
\end{align}
\end{theorem}
\noindent {\bf Proof of Estimate~\eqref{est:s1}.} 
First of all, with $ |D^2 p |^2=\displaystyle{ \sum_{i,j=1}^n (D_{ij}^2p )^2 }$, we recall the general identity
\begin{align}\label{eq:Fund1}
\int_{\R^n} | \nabla p |^2 \Delta p \, dx = \frac 23  \int_{\R^n} p |D^2 p |^2 \, dx  - \frac 23 \int_{\R^n} p |  \Delta p |^2 \, dx.
\end{align}

To begin, we multiply  Eq. \eqref{pressure}, by $-\big(\Delta p_m + \rho_m +G(p_m)\big)$ and integrate by parts. This gives
\begin{align*}
\frac 12 \frac{d}{dt} \int_{\R^n} | \nabla p_m(t) |^2 dx& +m \int_{\R^n} p_m \big(\Delta p_m + \rho_m +G(p_m)\big)^2  dx 
\\
&= \int_{\R^n} [- | \nabla p_m |^2 +\nabla p_m . \nabla \phi_m ] \big(\Delta p_m + \rho_m +G(p_m)\big).
\end{align*}
Therefore, using Eq. \eqref{eq:Fund1}, we get
\begin{align}\label{eq:grad}
\frac 12 \frac{d}{dt} \int_{\R^n} | \nabla p_m(t) |^2 dx& +m \int_{\R^n} p_m \big(\Delta p_m + \rho_m +G(p_m)\big)^2  dx +  \frac 23  \int_{\R^n} p_m |D^2 p_m |^2 \, dx  \notag
\\
&= \frac 23 \int_{\R^n} p_m |  \Delta p_m |^2 \, dx + \text{Rem} (t), 
\end{align}
where 
\begin{align*}
 \text{Rem} = - \int_{\R^n}  | \nabla p_m |^2 \big( \rho_m +  G(p_m)\big) + \int_{\R^n}  \nabla p_m . \nabla \phi_m \big(  \Delta p_m + \rho_m +G(p_m)\big).
\end{align*}

Our next step is to prove that 
\begin{align} \label{est:rem}
\int_0^T \text{Rem} (t) dt \leq C(T)+\frac 13    \int_0^T \int_{\R^n}  p_m  |D^2p_m|^2  dx .
\end{align}
That is because, in the definition of $Rem$,  the first integral is bounded thanks to the estimates in Proposition~\ref{thm:ape}. This is also true for the term 
$\nabla p_m . \nabla \phi_m \big( \rho_m  +G(p_m)\big)$. The only term which requires a treatment is 
\begin{align*}
 \int_{\R^n}  \nabla p_m . \nabla \phi_m   \Delta p_m dx&= - \int_{\R^n} \partial_i p_m [ \partial^2_{ij} \phi_m  \partial_j  p_m + \partial_{j} \phi_m \partial^2_{ij} p_m ] dx
 \\
 &= - \int_{\R^n} \partial_i p_m \partial^2_{ij} \phi_m  \partial_j  p_m + \frac 12 \int_{\R^n} \Delta \phi_m |\nabla p_m |^2] dx.
\end{align*}
After integration in time, the second term is bounded again  thanks to the estimates in Proposition~\ref{thm:ape}. We arrive at
\begin{align*}
 \int_0^T\int_{\R^n}  \nabla p_m . \nabla & \phi_m   \Delta p_m dxdt\leq C(T) +
  \int_0^T \int_{\R^n}  p_m [\partial^3_{iij} \phi_m  \partial_j  p_m + \partial^2_{ij} \phi_m \partial^2_{ij} p_m] dxdt
\\
&\leq  C(T) +   \int_0^T\int_{\R^n}  p_m \big[\partial_j\rho_m \partial_j  p_m +\frac 13  |D^2p_m|^2  +\frac 2 3  |D^2 \phi_m|^2 \big] dxdt.
\end{align*}
The first term is also $\frac{m}{m+1}\int_0^T \int_{\R^n}  \partial_j (p_m)^{\frac{m+1}{m}}  \partial_j  p_mdxdt$ and thus is bounded. The last term is also bounded because $ |D^2 \phi_m|^2$ has the same regularity in $L^p$, $1<p<\infty$ as $\rho_m$ according to  singular integral theory~\cite{Stein_book}. Therefore we have proved estimate~\eqref{est:rem}.
\\

We can now conclude the estimate~\eqref{est:s1}. After time integration of Eq.~\eqref{eq:grad} and using~\eqref{est:rem}, we find
\begin{align*}
\int_{\R^n}  | \nabla p_m(t) |^2 dx &+ m \int_0^T \int_{\R^n} p_m \big(\Delta p_m + \rho_m +G(p_m)\big)^2 dx dt 
 \\
 &+  \frac 13  \int_0^T\int_{\R^n} p_m |D^2 p_m |^2 \, dx \leq C(T)+\int_{\R^n}  | \nabla p_m^0 |^2 dx, 
\end{align*}
and the estimate~\eqref{est:s1} is proved.
\\

\noindent {\bf Proof of Estimate~\eqref{est:s2}.}
Integrating by parts, we may write
\begin{align*}
  \int_{\R^n} \frac{|\nabla p_m |^4 }{p_m^\al} dx  = -  \int_{\R^n} p_m \big[- \al  \frac{ |\nabla p_m |^4}{p_m^{\al+1}} + \frac{\Delta p_m}{p_m^\al} |\nabla p_m |^2+ 2 \frac{\partial^2_{ij} p_m \partial_{i} p_m \partial_{j} p_m  }{p_m^\al} \big]dx .
\end{align*}
Therefore, we also have
\begin{align*}
(1-\al) \int_{\R^n} \frac{|\nabla p_m |^4 }{p_m^\al} dx &=-  \int_{\R^n} p_m \big[ \frac{\Delta p_m}{p_m^\al} |\nabla p_m |^2+ 2 \frac{\partial^2_{ij} p_m \partial_{i} p_m \partial_{j} p_m  }{p_m^\al} \big]dx
\\
&\leq \frac{1-\al}{4} \int_{\R^n} \Big[\frac{|\nabla p_m |^4 }{p_m^\al}  +\frac{4} {1-\al}  p_m^{2-\al} |\Delta p_m|^2
\\
 &\qquad\qquad + \frac{1-\al}{4} \int_{\R^n} \frac{|\nabla p_m |^4 }{p_m^\al}+ \frac{16} {1-\al} \int_{\R^n} p_m^{2-\al} |D^2 p_m|^2\Big] dx.
\end{align*}
And finally, this is also written 
\begin{align*}
\frac{1-\al}{2} \int_{\R^n} \frac{|\nabla p_m |^4 }{p_m^\al} dx \leq \| p_m \|^{1-\al}_{L^\infty( \R^n)} \frac{C} {1-\al} \int_{\R^n} \big[ p_m |\Delta p_m|^2+p_m |D^2 p_m|\big] dx.
\end{align*}
We obtain \eqref{est:s2} because of the estimates \eqref{est:s1} which also furnish a uniform bound on $\int_0^T\int_{\R^n} p_m |\Delta p_m|^2 dx$.

%
\end{document}